\documentclass[12pt,reqno]{amsart}
\usepackage[a4paper,bindingoffset=0.1in,left=1in,right=1in,top=1in,bottom=1in,footskip=.25in]{geometry}
\usepackage{indentfirst,amssymb,amsmath,amsthm}    
\usepackage{newtxtext,newtxmath}
\usepackage{setspace}
\usepackage{times}
\usepackage[utf8]{inputenc}
\usepackage[T1]{fontenc}
\usepackage{verbatim}
\usepackage{hyperref}
\hypersetup{colorlinks=true, linkcolor=blue,citecolor=blue, urlcolor=blue}
\DeclareMathOperator{\ord}{ord}
\DeclareMathOperator{\dime}{dim}
\DeclareMathOperator{\pideg}{PI-deg}

\DeclareMathOperator{\cf}{Fract}
\DeclareMathOperator{\gcdi}{gcd}

\DeclareMathOperator{\diagonal}{diag}

\DeclareMathOperator{\kere}{ker}
\DeclareMathOperator{\ran}{rank}

\usepackage{mathtools}

\numberwithin{equation}{section}
 \newtheorem*{Theorem D}{Theorem D}
\newtheorem*{Theorem C}{Theorem C}
\newtheorem*{Theorem A}{Theorem A}
\newtheorem*{Theorem B}{Theorem B}
\newtheorem{theo}{Theorem}[section]
\newtheorem{defi}[theo]{Definition}
\newtheorem{lemm}[theo]{Lemma}
\newtheorem{rema}[theo]{Remark}
\newtheorem{coro}[theo]{Corollary}
\newtheorem{prop}[theo]{Proposition}

\begin{document}

\setcounter{page}{1} 
\baselineskip .65cm 
\pagenumbering{arabic}

\title[3-cyclic quantum Weyl Algebra]{Simple modules over 3-cyclic quantum Weyl Algebra at roots of unity}
\author [S Bera,~S Mandal,~S Mukherjee and~S Nandy]{Sanu Bera$^1$, Sugata Mandal$^2$, Snehashis Mukherjee$^3$, Soumendu Nandy$^4$}
\address {\newline  Sanu Bera$^1$, \newline Gandhi Institute of Technology and Management (GITAM), Hyderabad, \newline Rudraram, Patancheru mandal. Hyderabad-502329. Telangana, India. 
\newline Sugata Mandal$^2$, \newline School of Mathematical Sciences, \newline Ramakrishna Mission Vivekananda Educational and Research Institute (rkmveri), \newline Belur Math, Howrah-711202, West Bengal, India. 
\newline Snehashis Mukherjee$^3$, 
\newline Indian Institute of Technology (IIT), Kanpur, \newline Kalyanpur. Kanpur -208 016, Uttar Pradesh, India. 
\newline  Soumendu Nandy$^4$, \newline Ramakrishna Mission Vivekananda Educational and Research Institute (rkmveri), \newline Belur Math, Howrah-711202, West Bengal, India.
 }
\email{\href{mailto:sanubera6575@gmail.com}{sanubera6575@gmail.com$^1$}; \href{mailto:gmandal1961@gmail.com}{gmandal1961@gmail.com$^2$};\newline \href{mailto:tutunsnehashis@gmail.com}{tutunsnehashis@gmail.com$^3$};\href{mailto:soumendu.nandy100@gmail.com}{soumendu.nandy100@gmail.com$^4$};}

\begin{abstract}
This article undertakes an exploration of simple modules of 3-cyclic quantum Weyl algebra at roots of unity. Under the roots of unity assumption, the algebra becomes a Polynomial Identity algebra and the vector space dimension of the simple modules is bounded above by its PI degree. The article systematically classifies all potential simple modules and computes the algebra's center.
\end{abstract}
\maketitle
\section{Introduction}
Let $\mathbb{K}$ be a field and $\mathbb{K} ^{*}:=\mathbb{K}\setminus\{0\}$. The bi-quadratic algebra on $3$ generators is denoted by $\mathbb{K}[x_1, x_2, x_3; Q, \mathbb{A}, \mathbb{B}]$ and generated by $x_1,x_2,x_3$ subject to the following relations 
\begin{align*}
    x_2x_1-q_1x_1x_2&=ax_1+bx_2+cx_3+b_1\\
    x_3x_1-q_2x_1x_3&=\alpha x_1+\beta x_2+\gamma x_3+b_2\\
    x_3x_2-q_3x_2x_3&=\lambda x_1+\mu x_2+\nu x_3+b_3.
\end{align*}
 where $\mathbb{A}=\begin{pmatrix}
a&b&c\\
\alpha&\beta&\gamma\\
\lambda&\mu&\nu
\end{pmatrix}\in M_3(\mathbb{K}),\  Q = (q_1, q_2, q_3) \in (\mathbb{K}^*)^3\ \text{and}\ \mathbb{B}=(b_1,b_2,b_3)\in \mathbb{K}^3$. 
\par In 2023, V. Bavula introduced the algebra $\mathbb{K}[x_1, x_2, x_3; Q, \mathbb{A}, \mathbb{B}]$ for the first time in \cite{vbav1}. This groundbreaking work explicitly described bi-quadratic algebras on $3$ generators with a PBW basis. Notable examples of such algebras include the polynomial algebra in $3$ variables, the quantum affine space of rank $3$, and the universal enveloping algebra of the $3$-dimensional Heisenberg Lie algebra.
\par A vital class of bi-quadratic algebras on three generators is the $3$-cyclic quantum Weyl algebra \cite{vbav}. For $\alpha,\beta,\gamma\in \mathbb{K}$, the $3$-cyclic quantum Weyl algebra $\textit{A}(\alpha,\beta,\gamma)$ is the $\mathbb{K}$-algebra generated by $x,y$ and $z$ subject to the defining relations
\[xy=q^2yx+\alpha,\ \ xz=q^{-2}zx+\beta,\ \ yz=q^2zy+\gamma.\]
The algebra $\textit{A}(\alpha,\beta,\gamma)$ possesses an iterated skew polynomial representation structured as \[\mathbb{K}[x][y,\sigma_1;\delta_1][z,\sigma_2;\delta_2]\] where $\sigma_1$ denotes an automorphism of $\mathbb{K}[x]$, defined by $\sigma_1(x)=q^{-2}x$, and $\delta_1$ stands as a $\sigma_1$-derivation of $\mathbb{K}[x]$, with $\delta_1(x)=-q^{-2}\alpha$. Additionally, $\sigma_2$ represents an automorphism of $\mathbb{K}[x][y,\sigma_1]$, defined by $\sigma_2(x)=q^{2}x$ and $\sigma_2(y)=q^{-2}y$, while $\delta_2$ functions as a $\sigma_2$-derivation of $\mathbb{K}[x][y,\sigma_1]$, with $\delta_2(x)=-q^{2}\beta$ and $\delta_2(y)=-q^{-2}\gamma$. Consequently, the algebra $\textit{A}(\alpha,\beta,\gamma)$ constitutes a Noetherian domain with a Gelfand–Kirillov dimension of $3$. Moreover the monomials $x^ay^bz^c$ ($a,b,c\in \mathbb{Z}_{\geq 0}$) form a $\mathbb{K}$-basis for $\textit{A}(\alpha,\beta,\gamma)$.
\par  Ito, Terwilliger, and Weng \cite{itw} demonstrated that the algebra $U_q(\mathfrak{sl}_2)$ is a localization of the $3$-cyclic quantum Weyl algebra at powers of $x$, under certain choices of $\alpha, \beta,\gamma \in \mathbb{K}$. Specifically, they established that $U_q(\mathfrak{sl}_2)$ can be presented with generators $x^{\pm1}$, $y$, and $z$, subject to relations such as 
\[xx^{-1}=x^{-1}x=1,\ \displaystyle\frac{qxy-q^{-1}yx}{q-q^{-1}}=1,\ \displaystyle\frac{qyz-q^{-1}zy}{q-q^{-1}}=1, \ \displaystyle\frac{qzx-q^{-1}xz}{q-q^{-1}}=1.\]  This presentation, termed as \textit{equitable}, was shown to be isomorphic to the conventional presentation of $U_q(\mathfrak{sl}_2)$ \cite[Theorem 2.1]{itw}. Additionally, they illustrated that $y$ and $z$ lack invertibility within $U_q(\mathfrak{sl}_2)$ by constructing infinite-dimensional $U_q(\mathfrak{sl}_2)$-modules containing nontrivial null vectors for $y$ and $z$. However, their investigation of finite-dimensional $U_q(\mathfrak{sl}_2)$-modules concluded that $y$ and $z$ act invertibly on them.
\par In the study conducted by Bavula on the algebras denoted as $\textit{A}(\alpha,\beta,\gamma)$ under the condition that $q^2$ is not a root of unity, detailed insights are provided into their prime, completely prime, primitive, and maximal spectra along with the containment relations of prime ideals (within the Zariski–Jacobson topology on the spectrum) \cite[Theorem 5.2, Theorem 5.4, Theorem 7.4 and Theorem 7.6]{vbav}. Additionally, a thorough classification of simple $\textit{A}(\alpha,\beta,\gamma)$-modules is presented  \cite[Corollary 4.4, Corollary 5.5, Corollary 6.6, Corollary 6.9 and Corollary 6.1]{vbav}. For every prime ideal, an explicit set of generators is furnished.
 \par Moreover, it is demonstrated that the center $Z\left(\textit{A}(\alpha,\beta,\gamma)\right)$ of $\textit{A}(\alpha,\beta,\gamma)$ forms a polynomial algebra, being generated by a cubic element \cite[Theorem 2.3]{vbav} \[\Omega=yxz+\frac{q^{-2}\gamma}{q^2-1}x-\frac{q^2\beta}{q^2-1}y+\frac{\alpha}{q^2-1}z.\] A criterion for semisimplicity within the category of finite-dimensional $\textit{A}(\alpha,\beta,\gamma)$-modules is established \cite[Theorem 1.5]{vbav}. Criteria are also delineated for the commutativity of all ideals within the algebra $\textit{A}(\alpha,\beta,\gamma)$, as well as for each ideal to be a unique product of primes \cite[Theorem 1.4, Theorem 7.1]{vbav}.
 \par In this article, we aim to explore the algebra $\textit{A}(\alpha,\beta,\gamma)$ and its irreducible representations under the condition that $q^2$ is a root of unity. When $q^2 \neq 1$, the algebra $\textit{A}(\alpha,\beta,\gamma)$ transforms into a finitely generated module over its center within the roots of unity framework, thus constituting a Polynomial identity (PI) algebra. The theory of PI-algebras serves as a pivotal tool in analyzing this algebraic structure. The PI degree stands out as a fundamental invariant, constraining the $\mathbb{K}$-dimension of simple $\textit{A}(\alpha,\beta,\gamma)$-modules, with this constraint being indeed attained. The determination of the PI degree for $\textit{A}(\alpha,\beta,\gamma)$ holds significant importance as it provides insights into the classification of simple $\textit{A}(\alpha,\beta,\gamma)$-modules. 
 \par  Classifying simple modules for an infinite-dimensional non-commutative algebra poses a formidable challenge. Despite the significance of understanding the representation theory of quantum algebras, a method for this classification remains elusive. Notably, only a few examples exist where simple modules have been classified, primarily in the realm of generalized Weyl algebras (cf. \cite{ba1, ba2, ba3, ba5, ba6, ba7, ba8, ba9}) and Ore extensions with Dedekind rings as coefficient rings \cite{ba4}. Jordan achieved the classification of finite-dimensional simple $R$-modules for a specific class of iterated skew polynomial rings $R=A[y;\alpha][x;\alpha^{-1},\delta]$, where $A$ is an affine, commutative domain over an algebraically closed field \cite{dj}.
 \par Throughout the article $\mathbb{K}$ is an algebraically closed field of arbitrary characteristic and $q^2$ is a primitive $l(>1)$-th root of unity in $\mathbb{K}$.
 \par The article is organized as follows:
 \begin{itemize}
     \item Section $2$ presents key findings concerning the algebra $\textit{A}(\alpha,\beta,\gamma)$. Notably, it includes the computation of the algebra's center (refer to Theorem \ref{cen}).
     \item Section $3$ delves into preliminary results regarding Polynomial Identity algebras. Here, we establish that $\textit{A}(\alpha,\beta,\gamma)$ is PI algebra at root of unity (see Theorem \ref{finite}) and determine its PI degree (see Theorem \ref{pideg}).
     \item Section $5$ is dedicated to the classification of simple $\textit{A}(\alpha,\beta,\gamma)$-modules with invertible action of $z$ (see Theorem \ref{main1}). Importantly, this classification remains unaffected by the parameters $\alpha,\beta$, and $\gamma$. 
     \item In Section $6$ we proceed to construct and classify all simple $\textit{A}(\alpha,\beta,\gamma)$-modules with nilpotent action of $z$ under the assumption that $\beta\neq 0$ (see Theorem \ref{main2}).
     \item Finally, Sections $7$ and $8$ are devoted to the construction and classification of all simple $\textit{A}(\alpha,\beta,\gamma)$-modules with nilpotent action of $z$ when $\beta=0$ (refer to Theorems \ref{main3} and \ref{main4}). This completes the comprehensive classification of simple $\textit{A}(\alpha,\beta,\gamma)$-modules at root of unity.
 \end{itemize}
 \section{Some important results} This section is dedicated to exploring the fundamental identities within $\textit{A}(\alpha,\beta,\gamma)$, where our primary focus lies in uncovering pivotal identities essential for constructing and categorizing the simple $\textit{A}(\alpha,\beta,\gamma)$-modules. Additionally, we explore the center of $\textit{A}(\alpha,\beta,\gamma)$, providing valuable insights into its structure. 
 \par The following few results will be useful
 \begin{theo}
     The following identities hold in $\textit{A}(\alpha,\beta,\gamma)$:
     \begin{enumerate}
         \item $x^ay=q^{2a}yx^a+\displaystyle\frac{1-q^{2a}}{1-q^2}\alpha x^{a-1}$.
         \item $y^ax=q^{-2a}xy^a-q^{-2}\displaystyle\frac{1-q^{-2a}}{1-q^{-2}}\alpha y^{a-1}$.
         \item $x^az=q^{-2a}zx^a+\displaystyle\frac{1-q^{-2a}}{1-q^{-2}}\beta x^{a-1}$.
         \item  $z^ax=q^{2a}xz^a-q^{2}\displaystyle\frac{1-q^{2a}}{1-q^{2}}\beta z^{a-1}$.
         \item $y^az=q^{2a}zy^a+\displaystyle\frac{1-q^{2a}}{1-q^{2}}\gamma y^{a-1}$.
         \item $z^ay=q^{-2a}yz^a-q^{-2}\displaystyle\frac{1-q^{-2a}}{1-q^{-2}}\gamma z^{a-1}$.
     \end{enumerate}
 \end{theo}
 \begin{proof}
     The above identities can be obtained using induction on the defining relations of the algebra.
\end{proof}
\begin{coro}
If $q^2$ is a primitive $l$-th root of unity, then $x^l,y^l$ and $z^l$ are central elements in $\textit{A}(\alpha,\beta,\gamma)$.
\end{coro}
Let us consider a cubic element 
\begin{equation}\label{cenel}
\Omega:=yxz+\frac{q^{-2}\gamma}{q^2-1}x-\frac{q^2\beta}{q^2-1}y+\frac{\alpha}{q^2-1}z    
\end{equation}
in $\textit{A}(\alpha,\beta,\gamma)$ and this element belongs to the center of the algebra $\textit{A}(\alpha,\beta,\gamma)$ \cite[Theorem 2.3]{vbav}. Define \[d:=yx+\displaystyle\frac{\alpha}{q^2-1},\ \ e:=xz-\displaystyle\frac{q^2\beta}{q^2-1} \ \text{and}\ \ f:=yz+\displaystyle\frac{\gamma}{q^2-1}.\]
The element $\Omega$ can be written in the form 
\begin{equation}\label{omga}
    \Omega=zd+\displaystyle\frac{\gamma x}{q^2-1}-\displaystyle\frac{\beta y}{q^2-1}
\end{equation}
\begin{theo}\cite[eq 9]{vbav}
    The following commutation relations hold in $\textit{A}(\alpha,\beta,\gamma)$.
    \begin{enumerate}
    \item $ez=q^{-2}ze$ and $fz=q^2zf$.
        \item $f^ae=q^{2a}ef^a-q^{2a-2}\displaystyle\frac{1-q^{2a}}{1-q^2}\alpha z^2f^{a-1}+\displaystyle\frac{\beta\gamma q^2(q^{2a}-1)}{(q^2-1)^2}f^{a-1}$.
        \item $fe^a=q^{2a}e^af-\displaystyle\frac{q^{-2a}-1}{q^{-2}-1}\alpha z^2e^{a-1}+\displaystyle\frac{q^2\beta\gamma(q^{2a}-1)}{(q^2-1)^2} e^{a-1}.$
        
    \end{enumerate}
\end{theo}
 \subsection*{Center of $\textit{A}(\alpha,\beta,\gamma)$:} 
 Let $\mathcal{A}$ denote the quantum affine space generated by $X$, $Y$, and $Z$, subject to the relations:
 \[XY=q^2YX, \ XZ=q^{-2}ZX, \ YZ=q^2ZY.\] We refer to $\mathcal{A}$ as the associated quasipolynomial algebra of $\textit{A}(\alpha,\beta,\gamma)$. We define an ordering on the generators of $\textit{A}(\alpha,\beta,\gamma)$ and $\mathcal{A}$ by $x<y<z$ and $X<Y<Z$ respectively. Let $Z\left(\textit{A}(\alpha,\beta,\gamma)\right)$ be the center of $\textit{A}(\alpha,\beta,\gamma)$, and $\mathcal{Z}$ be the center of $\mathcal{A}$. For any $C \in Z\left(\textit{A}(\alpha,\beta,\gamma)\right)$, the leading term of $C$ must be of the form $kx^ay^bz^c$, where $k \in \mathbb{K}^*$ and $a,b,c \in \mathbb{Z}_{\geq 0}$. Since $xC=Cx$ holds, equating the coefficients of leading terms on both sides yields:
\begin{equation}\label{eq}
q^{2(c-b)}=1.
\end{equation}
Similarly, the equalities $yC=Cy$ and $zC=Cz$ give:
\begin{equation}\label{eq2}
q^{2(a-c)}=q^{2(b-a)}=1.
\end{equation}
From equations \eqref{eq} and \eqref{eq2}, we conclude that $kX^aY^bZ^c$ is a central element in $\mathcal{A}$.
  \par Hence we can define a map $T:Z\left(\textit{A}(\alpha,\beta,\gamma)\right) \rightarrow \mathcal{Z}$ such that \[T(C)=kX^aY^bZ^c\] where $kx^ay^bz^c$ is the leading term of $C\in Z\left(\textit{A}(\alpha,\beta,\gamma)\right)$. Note that $T$ is not linear and $T(0)=0$.
   \begin{prop}
Suppose that $q^2$ is a primitive $l$-th root of unity. Then the center $\mathcal{Z}$ of $\mathcal{A}$ is generated by $X^l,Y^l,Z^l$ and $XYZ$.  
 \end{prop}
 \begin{proof}
The skew-symmetric integral matrix associated with $\mathcal{A}$ based on the ordering of the generators is \[H:=\begin{pmatrix}
    0&1&-1\\
    -1&0&1\\
    1&-1&0
\end{pmatrix}.\]
We consider the matrix $H$ as a matrix of a homomorphism $H:\mathbb{Z}^3\rightarrow (\mathbb{Z}/{l\mathbb{Z}})^3$. Then the kernel of $H$ is
\[\ker(H)=\{(a,b,c)\in \mathbb{Z}^3:a\equiv b\equiv c~(\text{mod}~ l)\}.\] It is easy to verify that $\ker(H)\cap (\mathbb{Z}_{\geq 0})^3$ is a sub-semigroup of $\ker(H)$ generated by $(l,0,0)$, $(0,l,0),(0,0,l)$ and $(1,1,1)$. Then the assertion follows from \cite[Proposition 7.1(a)]{di}.
 \end{proof}
The central element $\Omega$ as in (\ref{cenel}) can be expressed using the ordering on the generators $x<y<z$ as \[\Omega:=q^{-2}xyz+\frac{q^{-2}\gamma}{q^2-1}x-\frac{q^2\beta}{q^2-1}y+\frac{q^{-2}\alpha}{q^2-1}z.\] From this expression, for any $n\geq 1$, we can compute
\begin{equation}\label{omgn}
    \Omega^n=(q^{-2})^{\frac{n(n+1)}{2}}x^ny^nz^n+ \text{lower degree terms}
\end{equation}
  \begin{theo}\label{cen}
      The center of the quantum 3-cyclic Weyl algebra $A(\alpha,\beta,\gamma)$ is generated by $x^l,y^l,z^l$ and $\Omega$. 
  \end{theo}
  \begin{proof}
     Let $Z_1$ denote the subalgebra of $Z\left(\textit{A}(\alpha,\beta,\gamma)\right)$ generated by the elements $x^l$, $y^l$, $z^l$ and $\Omega$. We use induction on the degree of the leading term to show that any $C \in Z\left(\textit{A}(\alpha,\beta,\gamma)\right)$ is also in $Z_1$.
Since $T(C)$ is a monomial belonging to $\mathcal{Z}$, it must be of the form \[T(C)=kX^{al}Y^{bl}Z^{cl}(XYZ)^d,\ \ \text{for}\ k \in \mathbb{K}^*.\] This implies that the leading term of $C$ will be $k'x^{al+d}y^{bl+d}z^{cl+d}$ for some $k'\in \mathbb{K}^*$. Then choose an element $C'$ in $Z_1$ of the form \[C':=(x^l)^a(y^l)^b(z^l)^c\Omega^d + \text{lower degree terms generated by $x^l,y^l,z^l,\Omega$}.\] 
By the equality (\ref{omgn}), the leading term of $C'$ becomes $k''x^{al+d}y^{bl+d}z^{cl+d}$ for some $k''\in \mathbb{K}^*$. Then the degree of the leading term of $k''C-k'C'$ is less than that of $C$. Hence by induction, $k''C-k'C' \in Z_1$. As $C' \in Z_1$, we have $C\in Z_1$. This completes the proof.
  \end{proof}
\section{Polynomial Identity Algebras}
In this section, we review relevant facts about the Polynomial Identity algebra and prove that $\textit{A}(\alpha,\beta,\gamma)$ is a PI algebra at root of unity. We then compute its PI degree, which enables us to comment on the $\mathbb{K}$-dimension of simple modules. 
\par A ring $R$ is said to be a Polynomial Identity (PI) ring if $R$ satisfies a monic polynomial $f\in \mathbb{Z}\langle x_1,\ldots,x_k\rangle$ i.e., $f(r_1,\ldots,r_k)=0$ for all $r_i\in R$.  The minimal degree of a PI ring $R$ is the least degree of all monic polynomial identities for $R$. PI rings cover a large class of rings including commutative rings. Commutative rings satisfy the polynomial identity $x_1x_2-x_2x_1$ and therefore have minimal degree $2$. Let us recall a result that provides a sufficient condition for a ring to be PI. 
\begin{prop}\emph{(\cite[Corollary 13.1.13]{mcr})}{\label{f}}
If $R$ is a ring that is a finitely generated module over a commutative subring, then $R$ is a PI ring.
\end{prop}
Now in the root of unity setting, the above proposition yields the following result.
\begin{theo} \label{finite}
 $\textit{A}(\alpha,\beta,\gamma)$ is a PI algebra if and only if $q^2(\neq 1)$ is a root of unity.
\end{theo}
\begin{proof}
Suppose $q^2$ is a primitive $l$-th root of unity. 
Then it is clear from Theorem \ref{cen} that the $\mathbb{K}[x^l,y^l,z^l]$ is a central subalgebra of $\textit{A}(\alpha,\beta,\gamma)$. This implies that $\textit{A}(\alpha,\beta,\gamma)$ becomes a finitely generated module over $\mathbb{K}[x^l,y^l,z^l]$. Therefore by Proposition \ref{f}, $\textit{A}(\alpha,\beta,\gamma)$ is a PI algebra.
\par For the converse part, if $q^2$ is not a root of unity, the subalgebra $\mathbb{K}\langle z,e: ez=q^2ze\rangle$ of $\textit{A}(\alpha,\beta,\gamma)$ is not PI algebra (cf. \cite[Proposition I.14.2.]{brg}).
\end{proof}
We now define the PI degree of prime PI-algebras. This definition will suffice because the algebra covered in this article is prime. We now recall one of the fundamental results from the Polynomial Identity theory.
\begin{theo}[Posner's Theorem {\cite[Theorem 13.6.5]{mcr}}]\label{pos}
 Let $A$ be a prime PI ring with center $Z(A)$ and minimal
degree $d$. Let $S=Z(A)\setminus\{0\}$, $Q=AS^{-1}$ and $F=Z(A)S^{-1}$. Then $Q$ is a central simple algebra with centre $F$ and $\dime_{F}(Q)=(\frac{d}{2})^2$.
\end{theo} 
\begin{defi}
The PI degree of a prime PI ring $A$ with minimal degree $d$ is $\pideg(A)=\frac{d}{2}$.
\end{defi}
The following result provides an important link between the PI degree of a prime affine PI algebra over an algebraically closed field and the $\mathbb{K}$-dimension of its irreducible representations (cf. \cite[Theorem I.13.5, Lemma III.1.2]{brg}). 
\begin{prop}\label{sim}
Let $A$ be a prime affine PI algebra over an algebraically closed field $\mathbb{K}$ and $V$ be a simple $A$-module. Then $V$ is a vector space over $\mathbb{K}$ of dimension $t$, where $t \leq \pideg (A)$. Moreover, the upper bound PI-deg($A$) is attained by simple $A$-modules. 
\end{prop}
From Theorem $\ref{finite}$ and Proposition \ref{sim}, we conclude that if $M$ is a simple $\textit{A}(\alpha,\beta,\gamma)$-module, then $\dim_{\mathbb{K}}M \leq \pideg(\textit{A}(\alpha,\beta,\gamma))$.
\par The following few results will be useful.
\begin{prop}\emph{\cite[Proposition 7.1]{di}}\label{quan}
Let $\mathbf{q}=\left(q_{ij}\right)$ be an  $n \times n$ multiplicatively antisymmetric matrix over $\mathbb{K}$.
 Suppose that $q_{ij}=q^{h_{ij}}$ for all $i,j$, where $q \in \mathbb{K}^*$ is a primitive $l$-th root of unity and the $h_{ij} \in \mathbb{Z}$. Let $h$ be the cardinality of the image of the homomorphism 
\[
    \mathbb{Z}^n \xrightarrow{(h_{ij})} \mathbb{Z}^n \xrightarrow{\pi} \left(\mathbb{Z}/l\mathbb{Z}\right)^n,
\]
where $\pi$ denotes the canonical epimorphism. Then \[\pideg(\mathcal{O}_{\mathbf{q}}(\mathbb{K}^n))=\pideg(\mathcal{O}_{\mathbf{q}}\left(\mathbb{(K^*)}^n\right)=\sqrt{h}.\]
\end{prop}
It is well known that a skew-symmetric matrix over $\mathbb{Z}$ such as our matrix $H:=(h_{ij})$ can be brought into a $2\times 2$ block diagonal form (commonly known as skew normal form) by an unimodular matrix $W\in GL_n(\mathbb{Z})$. Then $H$ is congruent to a block diagonal matrix of the form: 
\[WHW^{t}=\diagonal\left(\begin{pmatrix}
0&h_1\\
-h_1&0
\end{pmatrix},\cdots,\begin{pmatrix}
0&h_s\\
-h_s&0
\end{pmatrix},\bf{0}_{n-2s}\right),\]
where $\bf{0}_{n-2s}$ is the square matrix of zeros of dimension equals $\dime(\kere H)$, so that $2s=\ran(H)=n-\dime(\kere H)$ and $h_i\mid h_{i+1}\in \mathbb{Z}\setminus\{0\},~~\forall \ \ 1\leq i \leq s-1$. The nonzero $h_1,h_1,\cdots,h_s,h_s$ (each occurs twice) are called the invariant factors of $H$. The following result simplifies the calculation of $h$ in the statement of Proposition \ref{quan} by the properties of the integral matrix $H$, namely the dimension of its kernel, along with its invariant factors and the value of $l$.
\begin{lemm}\cite[Lemma 5.7]{ar}\label{mainpi}
Take $1\neq q\in \mathbb{K}^*$, a primitive $m$-th root of unity. Let $H$ be a skew symmetric integral matrix associated to $\mathbf{q}$ with invariant factors $h_1,h_1,\cdots,h_s,h_s$. Then PI degree of $\mathcal{O}_{\mathbf{q}}(\mathbb{K}^n)$ is given as \[\pideg(\mathcal{O}_{\mathbf{q}}(\mathbb{K}^n))=\pideg(\mathcal{O}_{\mathbf{q}}\left(\mathbb{(K^*)}^n\right)=\prod_{i=1}^{\frac{n-\dime(\kere H)}{2}}\frac{m}{\gcdi(h_i,m)}.\]
\end{lemm}
\subsection*{PI degree of $\textit{A}(\alpha,\beta,\gamma)$:} Now we aim to compute an explicit expression of PI-deg for $\textit{A}(\alpha,\beta,\gamma)$ in the following. First, we start by localizing $\textit{A}(\alpha,\beta,\gamma)$ with respect to the Ore set $S:=\{x^id^jq^{2k}: i,j \in \mathbb{N},k \in \mathbb{Z}\}$. Then from \cite[Proposition 2.1]{vbav} we have
\begin{equation}\label{spi}
    S^{-1}\textit{A}(\alpha,\beta,\gamma)=\mathbb{K}[\Omega] \otimes B
\end{equation} where $B=\mathbb{K}[d^{\pm1}][x^{\pm 1};\sigma]$ and $\sigma(d)=q^2d$. Observe that \[\textit{A}(\alpha,\beta,\gamma)\subset S^{-1}\textit{A}(\alpha,\beta,\gamma)\subset \cf \textit{A}(\alpha,\beta,\gamma)\] and \[\mathbb{K}[\Omega,d][x,\sigma] \subset S^{-1}\textit{A}(\alpha,\beta,\gamma) \subset \cf \left(\mathbb{K}[\Omega,d][x,\sigma]\right).\] Hence it follows from \cite[Corollary I.13.3]{brg} that \[\pideg(\textit{A}(\alpha,\beta,\gamma))=\pideg(S^{-1}\textit{A}(\alpha,\beta,\gamma))=\pideg\left(\mathbb{K}[\Omega,d][x,\sigma]\right).\] From Lemma \ref{mainpi}, we have \[\pideg(\textit{A}(\alpha,\beta,\gamma))=\pideg\left(\mathbb{K}[\Omega,d][x,\sigma]\right)=\displaystyle\frac{\ord(q)}{\gcd(2,\ord(q))}=\ord(q^2).\] We conclude this section with the following theorem.
\begin{theo}\label{pideg}
    Let $q^2$ be a primitive $l$-th root of unity with $l>1$. Then the PI degree of $\textit{A}(\alpha,\beta,\gamma)$ is given by $\pideg(\textit{A}(\alpha,\beta,\gamma))=l$. 
\end{theo}
\section{Classification of Simple $\textit{A}(\alpha, \beta, \gamma)$-modules}
Suppose that $\mathbb{K}$ is an algebraically closed field and $q^2$ is a primitive $l$-th root of unity in $\mathbb{K}$. In the root of unity setting, the algebra $\textit{A}(\alpha, \beta, \gamma)$ is classified as a prime affine PI algebra. Suppose $N$ be a simple $\textit{A}(\alpha, \beta, \gamma)$-module. Then by Proposition \ref{sim} along with Theorem \ref{pideg}, the $\mathbb{K}$-dimension of $N$ is finite and bounded above by the $\pideg \textit{A}(\alpha, \beta, \gamma)=l$. As the element $z^l$ is in the center of $\textit{A}(\alpha, \beta, \gamma)$, then by Schur's lemma, the action of $z$ on $N$ is either invertible or nilpotent. Based on this observation, the classification can be divided into the following cases:
\begin{enumerate}
    \item Simple $\textit{A}(\alpha, \beta, \gamma)$-modules with invertible action of $z$
    \item Simple $\textit{A}(\alpha, \beta, \gamma)$-modules with nilpotent action of $z$
\end{enumerate}
In the subsequent sections, we will focus on the classification process of each case.
\section{Simple $\textit{A}(\alpha, \beta, \gamma)$-modules with invertible action of $z$}
\subsection{Construction of simple $\textit{A}(\alpha, \beta, \gamma)$-modules}\label{s1} In this subsection we wish to construct simple $\textit{A}(\alpha, \beta, \gamma)$-modules with invertible action of $z$.
\par \textbf{Simple modules of type $M_1(\underline{\mu})$:} For $\underline{\mu}:=(\mu_1,\mu_2,\mu_3)\in \mathbb{({K}^*)}^2 \times \mathbb{K}$, let us consider the $\mathbb{K}$-vector space $M_1(\underline{\mu})$ with basis $\{m_{k}:0 \leq k \leq l-1\}$. Define an $\textit{A}(\alpha, \beta, \gamma)$-module structure on the $\mathbb{K}$-space $M_1(\underline{\mu})$ as follows: 
\begin{align*}
 m_kx&= \mu_1^{-1}\mu_2^{-1}q^{-2(k-1)}\left[q^{2k}\mu_3-\displaystyle\frac{q^{2k}-1}{q^2-1}\left(q^{2k-2}\alpha\mu_2^2-\frac{\beta\gamma q^2}{q^2-1}\right)\right]m_{k-1}+\displaystyle\frac{\mu_2^{-1}\beta q^2}{q^2-1}q^{-2k}m_k,\\
 m_ky&=\mu_1\mu_2^{-1}q^{-2(k+1)}m_{k+1}-\displaystyle\frac{\mu_2^{-1}\gamma}{q^2-1}q^{-2k}m_k\\
 m_kz&=q^{2k}\mu_2m_k.\end{align*}
   In order to establish the well-definedness we need to check that the $\mathbb{K}$-endomorphisms of $M_1(\underline{\mu})$ defined by the above rules satisfy the relations. Indeed with the above actions we have the following computations:
\begin{align*}
 &(m_k)xy\\
 &= \mu_1^{-1}\mu_2^{-1}q^{-2(k-1)}\left[q^{2k}\mu_3-\displaystyle\frac{q^{2k}-1}{q^2-1}\left(q^{2k-2}\alpha\mu_2^2-\frac{\beta\gamma q^2}{q^2-1}\right)\right]\left(m_{k-1}y\right)+\displaystyle\frac{\mu_2^{-1}\beta q^2}{q^2-1}q^{-2k}\left(m_ky\right)\\
 &=\displaystyle\frac{\mu_1\mu_2^{-2}\beta}{q^2-1}q^{-4k}m_{k+1}-\frac{\mu_1^{-1}\mu_2^{-2}\gamma q^{-4k+4}}{q^2-1}\left[q^{2k}\mu_3-\frac{q^{2k}-1}{q^2-1}\left(q^{2k-2}\alpha\mu_2^2-\frac{\beta\gamma q^2}{q^2-1}\right)\right]m_{k-1}\\
 &\ \ \ \ \ +\left[\mu_3\mu_2^{-2}q^{-2k+2}-\displaystyle\frac{1-q^{-2k}}{q^2-1}\alpha+\frac{\mu_2^{-2}\beta\gamma q^{-4k+2}}{(q^2-1)^2}\left(q^{4k}-2q^2-1\right)\right]m_k
\end{align*}
Again
\begin{align*}
    &(m_k)\left(q^2yx+\alpha\right)\\
    &=\mu_1\mu_2^{-1}q^{-2k}\left(m_{k+1}x\right)-\displaystyle\frac{\mu_2^{-1}\gamma}{q^2-1}q^{-2k+2}\left(m_kx\right)+\alpha m_k\\
  &=\displaystyle\frac{\mu_1\mu_2^{-2}\beta}{q^2-1}q^{-4k}m_{k+1}-\frac{\mu_1^{-1}\mu_2^{-2}\gamma q^{-4k+4}}{q^2-1}\left[q^{2k}\mu_3-\frac{q^{2k}-1}{q^2-1}\left(q^{2k-2}\alpha\mu_2^2-\frac{\beta\gamma q^2}{q^2-1}\right)\right]m_{k-1}\\
 &\ \ \ \ +\left[\mu_3\mu_2^{-2}q^{-2k+2}-\displaystyle\frac{1-q^{-2k}}{q^2-1}\alpha+\frac{\mu_2^{-2}\beta\gamma q^{-4k+2}}{(q^2-1)^2}\left(q^{4k}-2q^2-1\right)\right]m_k
\end{align*}
Similarly
\begin{align*}
(m_k)xz&=\mu_1^{-1}\mu_2^{-1}q^{-2(k-1)}\left[q^{2k}\mu_3-\displaystyle\frac{q^{2k}-1}{q^2-1}\left(q^{2k-2}\alpha\mu_2^2-\frac{\beta\gamma q^2}{q^2-1}\right)\right]\left(m_{k-1}z\right)\\& \ \ \ \ \ +\displaystyle\frac{\mu_2^{-1}\beta q^2}{q^2-1}q^{-2k}\left(m_kz\right)\\
&=\mu_1^{-1}\left[q^{2k}\mu_3-\displaystyle\frac{q^{2k}-1}{q^2-1}\left(q^{2k-2}\alpha\mu_2^2-\frac{\beta\gamma q^2}{q^2-1}\right)\right]m_{k-1}+\frac{\beta q^2}{q^2-1}m_k.
\end{align*}
Also 
\begin{align*}
    (m_k)\left(q^{-2}zx+\beta\right)&=q^{2k-2}\mu_2\left(m_kx\right)+\beta m_k\\
    &=\mu_1^{-1}\left[q^{2k}\mu_3-\displaystyle\frac{q^{2k}-1}{q^2-1}\left(q^{2k-2}\alpha\mu_2^2-\frac{\beta\gamma q^2}{q^2-1}\right)\right]m_{k-1}+\frac{\beta q^2}{q^2-1}m_k.
\end{align*}
Finally \[m_k\left(yz\right)=\mu_1\mu_2^{-1}q^{-2(k+1)}\left(m_{k+1}z\right)-\displaystyle\frac{\mu_2^{-1}\gamma}{q^2-1}q^{-2k}\left(m_kz\right)
    =\mu_1m_{k+1}-\displaystyle\frac{\gamma}{q^2-1}m_k.\]
and \[m_k\left(q^2zy+\gamma\right)=q^{2k+2}\mu_2\left(m_ky\right)+\gamma m_k=\mu_1m_{k+1}-\displaystyle\frac{\gamma}{q^2-1}m_k.\]
Thus $M_1(\underline{\mu})$ is an $\textit{A}(\alpha,\beta,\gamma)$-module.
 \begin{theo}\label{f1}
The module $M_1(\underline{\mu})$ is a simple $\textit{A}(\alpha,\beta,\gamma)$-module of dimension $l$.
\end{theo}
\begin{proof}
Let $P$ be a non-zero submodule of $M_1(\underline{\mu})$. We claim that $P$ contains a basis vector of the form $m_k$. Indeed, any member $p\in P$ is a finite $\mathbb{K}$-linear combination of such vectors. i.e.,
\[
  p:=\sum_{\text{finite}} \lambda_k~m_k  
\]
for some $\lambda_k\in \mathbb{K}$. Suppose there exist two non-zero coefficients, say, $\lambda_u,\lambda_v$ with $1\leq u,v \leq l-1$. Then in particular focus on the two distinct vectors $m_u$ and $m_v$.
\par Then the vectors $m_u$ and $m_v$ are eigenvectors of $z$ associated with the eigenvalues
$\Lambda_u=\mu_2q^{2u}$ and $\Lambda_v=\mu_2q^{2v}$ respectively. In fact $\Lambda_u$ and $\Lambda_v$ are distinct. Indeed,
\begin{equation}\label{coo}
    \Lambda_u=\Lambda_v\implies l \mid (u-v),
\end{equation} which is a contradiction. 
\par Now $pz-\Lambda_up$ is a nonzero element in $P$ of smaller length than $p$. Hence by induction, it follows that every nonzero submodule of $M_1(\underline{\mu})$ contains a basis vector of the form $m_k$. Thus $M_1(\underline{\mu})$ is a simple $\textit{A}(\alpha,\beta,\gamma)$-module by the actions of $x,y$. and $z$.
\end{proof}
The following theorem provides a criterion for determining when two distinct modules of this kind are isomorphic.
\begin{theo}
    Let $\underline{\mu}=(\mu_1,\mu_2,\mu_3)$ and $\underline{\mu}'=(\mu'_1,\mu'_2,\mu'_3)$ belong to ${(\mathbb{K}^*)}^2 \times \mathbb{K}$. Then $M_1(\underline{\mu})\cong M_1(\underline{\mu}')$ if and only if $\mu_1^l={\mu'_1}^l$ and there exists some $i$ with $0 \leq i \leq l-1$ such that 
    \[\mu_2=q^{2i}\mu'_2\ \ \text{and}\ \ 
         \mu_3=q^{2i}\mu'_3+\displaystyle\frac{q^{2i}-1}{q^2-1}\left(\frac{\beta\gamma q^2}{q^2-1}-q^{2i-2}\alpha{\mu'_2}^2\right).\]
\end{theo}
\begin{proof}
    Let $\phi:M_1(\underline{\mu})\rightarrow M_1(\underline{\mu}')$ be a module homomorphism. Note that as 
    \[m_k=\mu_1^{-k}m_0f^k \in M_1(\underline{\mu}),\] the homomorphism $\phi$ can be uniquely determined by $\phi(m_0)$. Suppose
    \[\phi(m_0)=\sum_{p \in \mathcal{I}}\lambda_pm_p\] where $\mathcal{I}$ is a finite subset of $\{0,\ldots,l-1\}$ with $\lambda_p\neq 0$ for all $p\in\mathcal{I}$. Now the equality $\phi(m_0z)=\phi(m_0)z$ gives $\mu_2-\mu'_2q^{2p}=0$ for all $\ p \in \mathcal{I}$. 
This implies the index set $\mathcal{I}$ must be singleton. 
Therefore $\phi(m_0)=\lambda_im_i$ for some $i$ with $0 \leq i \leq l-1$.
Then the equality $\phi(m_0z)=\phi(m_0)z$ gives $\mu_2=q^{2i}\mu'_2$. Similarly the equalities $\phi(m_0f^l)=\phi(m_0)f^l$ and $\phi(m_0ef)=\phi(m_0)ef$ give $\mu_1^l={\mu'_1}^l$ and  $\mu_3=q^{2i}\mu'_3+\displaystyle\frac{q^{2i}-1}{q^2-1}\left(\frac{\beta\gamma q^2}{q^2-1}-q^{2i-2}{\alpha\mu'_2}^2\right)$ respectively.
\par Conversely assume the relations between $\underline{\mu}=(\mu_1,\mu_2,\mu_3)$ and $\underline{\mu'}=(\mu'_1,\mu'_2,\mu'_3)$ holds. Then the $\mathbb{K}$-linear map $\psi:M_1(\underline{\mu})\rightarrow M_1(\underline{\mu}')$ defined by $\psi(m_k)={(\mu_1^{-1}\mu'_1)}^km_{k\oplus i}$ where $\oplus$
denotes the addition modulo $l$ is an $\textit{A}(\alpha,\beta,\gamma)$-module isomorphism.
\end{proof}
\par \textbf{Simple modules of type $M_2(\underline{\mu})$:} For $\underline{\mu}:=(\mu_1,\mu_2)\in \mathbb{({K}^*)}^2$, let us consider the $\mathbb{K}$-vector space $M_2(\underline{\mu})$ with basis $\{m_{k}:0 \leq k \leq l-1\}$. Define an $\textit{A}(\alpha, \beta, \gamma)$-module structure on the $\mathbb{K}$-space $M_2(\underline{\mu})$ as follows: 
\begin{align*}
m_kx&=\mu_1\mu_2^{-1}q^{2(k+1)}m_{k+1}+\displaystyle\frac{\beta\mu_2^{-1}q^2}{q^2-1}q^{2k}m_k\\
 m_ky&=\begin{cases}\mu_1^{-1}\mu_2^{-1}q^{2(k-1)}\displaystyle\frac{1-q^{-2k}}{1-q^{-2}}\left[\alpha\mu_2^2q^{-2k}-\frac{\beta\gamma}{q^2-1}\right]m_{k-1}-\frac{\gamma\mu_2^{-1}}{q^2-1}q^{2k}m_k,& k \neq 0\\
    -\displaystyle\frac{\gamma\mu_2^{-1}}{q^2-1}q^{2k}m_k,& k= 0 
 \end{cases}\\
 m_kz&=q^{-2k}\mu_2m_k.\end{align*}
 We can readily confirm that the action above defines an $\textit{A}(\alpha,\beta,\gamma)$-module, similarly as detailed in Subsection \ref{s1}. 
 \begin{theo}\label{f2}
The module $M_2(\underline{\mu})$ is a simple $\textit{A}(\alpha,\beta,\gamma)$-module of dimension $l$.
\end{theo}
\begin{proof}
Note that the vector $m_k\in M_2(\underline{\mu})$ is an eigenvector of $z$ associated with the eigenvalues $\mu_2q^{-2k}$. With this fact, the proof is parallel to the proof of Theorem $\ref{f1}$.  
\end{proof}
The following theorem offers a criterion to ascertain the isomorphism between two distinct modules of this type.
\begin{theo}
    Let $\underline{\mu}=(\mu_1,\mu_2)$ and $\underline{\mu}'=(\mu'_1,\mu'_2)$ belong to ${(\mathbb{K}^*)}^2$. Then $M_2(\underline{\mu})\cong M_2(\underline{\mu}')$ if and only if $\mu_1^l={\mu'_1}^l$ and $\mu_2=\mu'_2$. 
\end{theo}
 \textbf{Simple modules of type $M_3({\mu},s)$:} For ${\mu}\in \mathbb{{K}^*}$, define \[F_{\alpha,\beta,\gamma}(k,\mu):=\alpha q^{-2k}\mu^2-\frac{\beta\gamma}{q^2-1},\ \ 1\leq k\leq l\] and \[s:=\begin{cases}
     1,&\alpha=\beta\gamma=0\\
     k,&\alpha,\beta,\gamma\in\mathbb{K}^*\ \text{and}\  F_{\alpha,\beta,\gamma}(k,\mu)=0\\
     l,& \text{otherwise}
 \end{cases}\] Consider the $\mathbb{K}$-vector space $M_3({\mu},s)$ with basis $\{m_{k}:0 \leq k \leq s-1\}$. Define an $\textit{A}(\alpha, \beta, \gamma)$-module structure on the $\mathbb{K}$-space $M_3({\mu},s)$ as follows: 
\begin{align*}
m_kx&=\begin{cases}
    \mu^{-1}q^{2(k+1)}m_{k+1}+\displaystyle\frac{\beta\mu^{-1}}{q^2-1}q^{2k+2}m_k,& k\neq s-1\\
    \displaystyle\frac{\beta\mu^{-1}}{q^2-1}q^{2k+2}m_k,& k= s-1
\end{cases}\\
m_ky&=\begin{cases}\mu^{-1}q^{2(k-1)}\displaystyle\frac{1-q^{-2k}}{1-q^{-2}}\left[\alpha\mu^2q^{-2k}-\frac{\beta\gamma}{q^2-1}\right]m_{k-1}-\frac{\gamma\mu^{-1}}{q^2-1}q^{2k}m_k,& k \neq 0\\
    -\displaystyle\frac{\gamma\mu^{-1}}{q^2-1}q^{2k}m_k,& k= 0 
 \end{cases}\\
 m_kz&=q^{-2k}\mu m_k.
\end{align*} The above action indeed defines an $\textit{A}(\alpha,\beta,\gamma)$-module.
\begin{theo}\label{f3.1}
The module $M_3({\mu},s)$ is a simple $\textit{A}(\alpha,\beta,\gamma)$-module of dimension $s$.
\end{theo}
\begin{proof}
Note that the vector $m_k\in M_3({\mu},s)$ is an eigenvector of $z$ associated with the eigenvalues $\mu q^{-2k}$. With this fact, the proof is parallel to the proof of Theorem $\ref{f1}$.  
\end{proof}
Finally, we have the following theorem determining the isomorphism class of modules of type $M_3({\mu},s)$
\begin{theo}
    Let ${\mu}$ and $\mu'$ belong to ${\mathbb{K}^*}$. Then $M_3({\mu},s)\cong M_3({\mu}',s)$ if and only if $\mu=\mu'$. 
\end{theo}
\begin{rema} We have the following observations.
\begin{enumerate}
\item $f^l$ does not annihilate the simple $\textit{A}(\alpha,\beta,\gamma)$-module $M_1(\underline{\mu})$.
\item $f^l$ annihilates the simple $\textit{A}(\alpha,\beta,\gamma)$-module $M_2(\underline{\mu})$, but $e^l$ does not.
\item Both $f^l$ and $e^l$ annihilate the simple $\textit{A}(\alpha,\beta,\gamma)$-module $M_3({\mu},s)$.
\end{enumerate}
Thus the three types of simple $\textit{A}(\alpha,\beta,\gamma)$-modules are non-isomorphic.
\end{rema}
\subsection{{Classification of simple $\textit{A}(\alpha,\beta,\gamma)$-modules}}
This subsection aims to classify simple $\textit{A}(\alpha,\beta,\gamma)$-modules with invertible action of $z$. 
Suppose $N$ is a simple module over $\textit{A}(\alpha,\beta,\gamma)$ with $Nz\neq 0$. 
Observe that each of the elements 
\begin{equation}\label{eq1}
 f^l,\ e^l,\ ef,\ z   
\end{equation}
commute in $\textit{A}(\alpha,\beta,\gamma)$. Since $N$ is finite-dimensional, there is a common eigenvector $v$ in $N$ of the operators (\ref{eq1}) such that 
\begin{equation}\label{coev}
vf^l=\eta_1v,\ ve^l=\eta_2v,\ vef=\eta_3v,\ vz=\eta_4v    
\end{equation}
for some $\eta_1,\eta_2,\eta_3 \in \mathbb{K}$ and $\eta_4\in\mathbb{K}^*$. 
Then the following cases arise depending on the scalars $\eta_1$ and $\eta_2$.\\
\textbf{Case I:} Suppose that $\eta_1 \neq 0$. Then the set $\{vf^k:0 \leq k \leq l-1\}$ consists of nonzero vectors of $N$. Let us choose
\[\mu_1:=\eta_1^{\frac{1}{l}},\ \mu_2:=\eta_4,\ \mu_3:=\eta_3 \] so that 
$\underline{\mu}=(\mu_1,\mu_2,\mu_3)\in {(\mathbb{K}^*)}^2 \times \mathbb{K}$. Now define a $\mathbb{K}$-linear map \[\Phi_1:M_1(\underline{\mu})\rightarrow N\] by specifying the image of the basis vectors of $M_1(\underline{\mu})$ only
\[\Phi_1(m_k):=\mu_1^{-k}vf^k\]One can easily verify that $\Phi_1$ is a non zero $\textit{A}(\alpha,\beta,\gamma)$-module homomorphism. In this verification the following computations will be very useful:
\begin{align*}
\left(vf^k\right)x&=\eta_4^{-1}q^{-2k}(vf^k)zx\\&=\eta_4^{-1}q^{-2k}\left(vf^k\right)\left(q^2e+\displaystyle\frac{q^2\beta}{q^2-1}\right)\\&=\begin{cases}
q^{-2(k-1)}\eta_4^{-1}\left[q^{2k}\eta_3-\displaystyle\frac{q^{2k}-1}{q^2-1}\left(q^{2k-2}\alpha\eta_4^2-\displaystyle\frac{\beta\gamma q^2}{q^2-1}\right)\right]\left(vf^{k-1}\right)\\ \hspace{2cm}+\displaystyle\frac{q^2\beta}{q^2-1}q^{-2k}\eta_4^{-1}\left(vf^k\right),&k \neq 0 \\
 \eta_1^{-1}\eta_4^{-1}\eta_3q^2\left(vf^{l-1}\right)+\displaystyle\frac{q^2\beta}{q^2-1}\eta_4^{-1}\left(v\right),&k= 0
\end{cases}\\
\left(vf^k\right)y&=\eta_4^{-1}q^{-2k}(vf^k)zy\\&=\eta_4^{-1}q^{-2k-2}\left(vf^k\right)\left(f-\displaystyle\frac{q^2\gamma}{q^2-1}\right)\\&=\begin{cases}
\eta_4^{-1}q^{-2(k+1)}\left(vf^{k+1}\right)-\displaystyle\frac{\gamma}{q^2-1}\eta_4^{-1}q^{-2k}\left(vf^k\right),& k\neq l-1\\
  \eta_1\eta_4^{-1}\left(v\right)-\displaystyle\frac{q^2\gamma}{q^2-1}\eta_4^{-1}\left(vf^{l-1}\right),& k=l-1.
  \end{cases}
\end{align*}
Thus by Schur's lemma, $\Phi_1$ is an isomorphism because $M_1(\underline{\mu})$ and $N$ are both simple $\textit{A}(\alpha,\beta,\gamma)$-modules.\\
\textbf{Case II:} $\eta_1=0$ and $\eta_2 \neq 0$. Since $\eta_1 =0$ there exists $0\leq p\leq l-1$ such that $vf^{p}\neq 0$ and $vf^{p+1}=0$. Define $u:=vf^{p}\neq 0$. Then from (\ref{coev}) we obtain \[ue^l=\eta_2u, \ uz=q^{2p}\eta_4u.\] Note that as $\eta_2\neq 0$, the set $\{ue^k:0 \leq k \leq l-1\}$ consists of nonzero vectors of $N$. Let us choose
\[\mu_1:=\eta_2^{\frac{1}{l}},\ \mu_2:=q^{2p}\eta_4 \] so that 
$\underline{\mu}=(\mu_1,\mu_2)\in {(\mathbb{K}^*)}^2 $. Now define a $\mathbb{K}$-linear map \[\Phi_2:M_2(\underline{\mu})\rightarrow N\] by specifying the image of the basis vectors of $M_2(\underline{\mu})$ only
\[\Phi_2(m_k):=\mu_1^{-k}ue^k\]One can easily verify that $\Phi_2$ is a non-zero $\textit{A}(\alpha,\beta,\gamma)$-module homomorphism. In this verification the following computations will be very useful:
\begin{align*}
\left(ue^k\right)x&=(q^{2p}\eta_4)^{-1}q^{2k}(ue^k)zx\\&=(q^{2p}\eta_4)^{-1}q^{2k}\left(ue^k\right)\left(q^2e+\displaystyle\frac{q^2\beta}{q^2-1}\right)\\&=\begin{cases}
q^{2(k+1)}q^{-2p}\eta_4^{-1}\left(ue^{k+1}\right)+\displaystyle\frac{q^2\beta}{q^2-1}q^{2k}q^{-2p}\eta_4^{-1}\left(ue^k\right),&k \neq l-1 \\
 \eta_2\eta_4^{-1}q^{-2p}\left(u\right)+\displaystyle\frac{\beta}{q^2-1}q^{-2p}\eta_4^{-1}\left(ue^{l-1}\right),&k= l-1.
\end{cases}\\
\left(ue^k\right)y&=(q^{2p}\eta_4)^{-1}q^{2k}(ue^k)zy\\&=(q^{2p}\eta_4)^{-1}q^{2k-2}\left(ue^k\right)\left(f-\displaystyle\frac{q^2\gamma}{q^2-1}\right)\\&=\begin{cases}
q^{2(k-1)}q^{-2p}\eta_4^{-1}\displaystyle\frac{1-q^{-2k}}{1-q^{-2}}\left[q^{-2k}q^{4p}\eta_4^2\alpha-\frac{\beta\gamma}{q^2-1}\right]\left(ue^{k-1}\right)\\ \hspace{2cm}-\displaystyle\frac{\gamma}{q^2-1}q^{2k}q^{-2p}\eta_4^{-1}\left(ue^k\right),& k\neq 0\\
-\displaystyle\frac{\gamma}{q^2-1}q^{-2p}\eta_4^{-1}\left(u\right),& k=0.
  \end{cases}
\end{align*}
Again by Schur's lemma, $\Phi_2$ is an isomorphism because $M_2(\underline{\mu})$ and $N$ are both simple $\textit{A}(\alpha,\beta,\gamma)$-modules.\\
\textbf{Case III:}  $\eta_1=0$ and $\eta_2 =0$. Since $\eta_1 =0$ there exists $0\leq p\leq l-1$ such that $vf^{p}\neq 0$ and $vf^{p+1}=0$. Define $u:=vf^{p}\neq 0$. Then form (\ref{coev}) we compute
\[ue^l=\eta_2u, \ uz=q^{2p}\eta_4u.\]
Suppose $k$ be the smallest integer with $1 \leq k \leq l$ such that $ue^{k-1} \neq 0$ and $ue^k=0$. Then after simplifying the equality $(ue^k)f=0$ we have
\begin{align*}
      0=(ue^k)f&=u\Big( q^{-2k}fe^k+q^{-2k}\displaystyle\frac{1-q^{-2k}}{1-q^{-2}}\alpha z^2e^{k-1}-\frac{\beta\gamma}{q^2-1}\frac{1-q^{-2k}}{1-q^{-2}}e^{k-1}\Big)\\
    &=\displaystyle\frac{1-q^{-2k}}{1-q^{-2}}\left(\alpha q^{-2k}q^{4p}\eta_4^2-\frac{\beta\gamma}{q^2-1}\right)ue^{k-1}.
    \end{align*}
This implies \[(1-q^{-2k})F_{\alpha,\beta,\gamma}(k,q^{4p}\eta_4^2)=0.\] Therefore we set \[s:=\begin{cases}
     1,&\alpha=\beta\gamma=0\\
     k,&\alpha,\beta,\gamma\in\mathbb{K}^*\ \text{and}\  F_{\alpha,\beta,\gamma}(k,q^{4p}\eta_4^2)=0\\
     l,& \text{otherwise}
\end{cases}\]
For such $1\leq s\leq l$, it follows that the vectors $ue^r$, where $0 \leq r\leq s-1$ of $N$ are non-zero. Take $\mu:=q^{2p}\eta_4 \in \mathbb{K}^* $. Now define a $\mathbb{K}$-linear map \[\Phi_3:M_3({\mu},s)\rightarrow N\] by specifying the image of the basis vectors of $M_3({\mu},s)$ only
\[\Phi_3(m_k):=ue^k,\ 0\leq k\leq s-1.\]
One can easily verify that $\Phi_3$ is a non-zero $\textit{A}(\alpha,\beta,\gamma)$-module homomorphism. In this verification the following computations will be very useful:
\begin{align*}
\left(ue^k\right)x&=(q^{2p}\eta_4)^{-1}q^{2k}(ue^k)zx\\&=(q^{2p}\eta_4)^{-1}q^{2k}\left(ue^k\right)\left(q^2e+\displaystyle\frac{q^2\beta}{q^2-1}\right)\\&=\begin{cases}
q^{2(k+1)}q^{-2p}\eta_4^{-1}\left(ue^{k+1}\right)+\displaystyle\frac{q^{2k+2}\beta}{q^2-1}q^{-2p}\eta_4^{-1}\left(ue^k\right),&k \neq s-1 \\
\displaystyle\frac{q^{2k+2}\beta}{q^2-1}q^{-2p}\eta_4^{-1}\left(ue^{k}\right),&k= s-1.
\end{cases}\\
\left(ue^k\right)y&=(q^{2p}\eta_4)^{-1}q^{2k}(ue^k)zy\\&=(q^{2p}\eta_4)^{-1}q^{2k-2}\left(ue^k\right)\left(f-\displaystyle\frac{q^2\gamma}{q^2-1}\right)\\&=\begin{cases}
q^{2(k-1)}q^{-2p}\eta_4^{-1}\displaystyle\frac{1-q^{-2k}}{1-q^{-2}}\left[q^{-2k}q^{4p}\eta_4^2\alpha-\frac{\beta\gamma}{q^2-1}\right]\left(ue^{k-1}\right)\\ \hspace{2cm}-\displaystyle\frac{\gamma}{q^2-1}q^{2k}q^{-2p}\eta_4^{-1}\left(ue^k\right),& k\neq 0\\
-\displaystyle\frac{\gamma}{q^2-1}q^{-2p}\eta_4^{-1}\left(u\right),& k=0.
  \end{cases}
\end{align*}
Again by Schur's lemma, $\Phi_3$ is an isomorphism because $M_3({\mu},s)$ and $N$ are both simple $\textit{A}(\alpha,\beta,\gamma)$-modules.
\par Finally the above discussions lead us to one of the main results of this section which provides an opportunity for the classification of simple $\textit{A}(\alpha,\beta,\gamma)$-modules in terms of scalar parameters.
\begin{theo}\label{main1}
Let $q^2$ be a primitive $l$-th root of unity. Then each simple $\textit{A}(\alpha,\beta,\gamma)$-module with invertible action of $z$ is isomorphic to one of the following simple $\textit{A}(\alpha,\beta,\gamma)$-modules:
\begin{itemize}
    \item $M_1(\underline{\mu})$ for some $\underline{\mu}:=(\mu_1,\mu_2,\mu_3)\in \mathbb{({K}^*)}^2 \times \mathbb{K}$,
    \item $M_2(\underline{\mu})$ for some $\underline{\mu}:=(\mu_1,\mu_2)\in \mathbb{({K}^*)}^2$,
    \item $M_3({\mu},s)$ for some ${\mu}\in \mathbb{{K}^*}$ and $1\leq s\leq l$.
\end{itemize}
\end{theo}
\section{Simple $\textit{A}(\alpha, \beta, \gamma)$-modules with nilpotent action of $z$ and $\beta \neq 0$.}
\subsection{Construction of simple $\textit{A}(\alpha,\beta,\gamma)$-modules} In this subsection we construct simple $\textit{A}(\alpha,\beta,\gamma)$-modules with nilpotent action of $z$ and with the assumption $\beta \in \mathbb{K}^*$.
\par\textbf{Simple modules of type $M_4(\underline{\mu})$:} For $\underline{\mu}:=(\mu_1,\mu_2) \in \mathbb{K}^2$, let us consider the $\mathbb{K}$-vector space $M_4(\underline{\mu})$ with basis $\{m_{k}:0 \leq k \leq l-1\}$. Define an $\textit{A}(\alpha, \beta, \gamma)$-module structure on the $\mathbb{K}$-space $M_4(\underline{\mu})$ as follows: 
\begin{align*}
m_kx&=\begin{cases}
    m_{k+1},& k\neq l-1\\
    \mu_1m_0,& k=l-1
\end{cases}\\
m_ky&=\begin{cases}\gamma\beta^{-1}q^{2k}m_{k+1}-\mu_2\beta^{-1}q^{2k}(q^2-1)m_k+\displaystyle\frac{1-q^{2k}}{1-q^2}\alpha m_{k-1},&k \neq 0,l-1\\
\gamma\beta^{-1}\mu_1q^{2k}m_{0}-\mu_2\beta^{-1}q^{2k}(q^2-1)m_k+\displaystyle\frac{1-q^{2k}}{1-q^2}\alpha m_{k-1},&k =l-1\\
    \gamma\beta^{-1}m_1-\mu_2\beta^{-1}(q^2-1)m_0,&k= 0.
\end{cases}\\
m_kz&=\begin{cases}
    \beta\displaystyle\frac{1-q^{-2k}}{1-q^{-2}}m_{k-1},& k\neq 0\\
    0,& k=0.
\end{cases}
\end{align*}
To ensure well-definedness, we must verify that the $\mathbb{K}$-endomorphisms of $M_4(\underline{\mu})$ defined by the aforementioned rules adhere to the prescribed relations. Through the actions described above, we undertake the following computations:\\
For $k \neq 0,l-1$, we have
\begin{align*}
    (m_k)xy&=m_{k+1}y\\&=\begin{cases}
      \gamma\beta^{-1}q^{2k+2}m_{k+2}-\mu_2\beta^{-1}q^{2k+2}(q^2-1)m_{k+1}+\displaystyle\frac{1-q^{2k+2}}{1-q^2}\alpha m_{k},&k \neq l-2\\
      \gamma\beta^{-1}\mu_1q^{2k+2}m_{0}-\mu_2\beta^{-1}q^{2k+2}(q^2-1)m_{k+1}+\displaystyle\frac{1-q^{2k+2}}{1-q^2}\alpha m_{k},&k =l-2.
    \end{cases} 
    \end{align*}
    Similarly \begin{align*}
       m_k\left(q^2yx+\alpha\right)&=q^2\left(\gamma\beta^{-1}q^{2k}m_{k+1}-\mu_2\beta^{-1}q^{2k}(q^2-1)m_k+\displaystyle\frac{1-q^{2k}}{1-q^2}\alpha m_{k-1}\right)x+\alpha m_k\\
         &=\begin{cases}
           \gamma\beta^{-1}q^{2k+2}m_{k+2}-\mu_2\beta^{-1}q^{2k+2}(q^2-1)m_{k+1}+\displaystyle\frac{1-q^{2k+2}}{1-q^2}\alpha m_{k},&k \neq l-2\\
           \gamma\beta^{-1}\mu_1q^{2k+2}m_{0}-\mu_2\beta^{-1}q^{2k+2}(q^2-1)m_{k+1}+\displaystyle\frac{1-q^{2k+2}}{1-q^2}\alpha m_{k},&k =l-2.
         \end{cases}
    \end{align*}
For $k=0$, we have 
\[(m_0)xy=m_{1}y=\begin{cases}
      \gamma\beta^{-1}q^{2}m_{2}-\mu_2\beta^{-1}q^{2}(q^2-1)m_{1}+\alpha m_{0},& l\neq 2\\
      \gamma\beta^{-1}\mu_1q^{2}m_{0}-\mu_2\beta^{-1}q^{2}(q^2-1)m_{1}+\alpha m_{0},&l=2. \end{cases} \]
      On the other hand
      \begin{align*}
      m_0\left(q^2yx+\alpha\right)
      &= \left(\gamma\beta^{-1}q^2m_1-\mu_2\beta^{-1}q^2(q^2-1)m_0\right)x+\alpha m_0\\
      &=\begin{cases}
        \gamma\beta^{-1}q^2m_2 -\mu_2\beta^{-1}q^2(q^2-1)m_1+\alpha m_0, & l \neq 2\\
        \gamma\beta^{-1}\mu_1q^{2}m_{0}-\mu_2\beta^{-1}q^{2}(q^2-1)m_{1}+\alpha m_{0},&l=2.
      \end{cases}\end{align*}
      For $k=l-1$
      \[m_{l-1}(xy)=\mu_1m_0y=\mu_1\gamma\beta^{-1}m_1-\mu_1\mu_2\beta^{-1}(q^2-1)m_0\]
      On other hand
      \begin{align*}
      &m_{l-1}\left(q^2yx+\alpha\right)\\&=q^2\left(\gamma\beta^{-1}\mu_1q^{2l-2}m_{0}-\mu_2\beta^{-1}q^{2l-2}(q^2-1)m_{l-1}+\displaystyle\frac{1-q^{2l-2}}{1-q^2}\alpha m_{l-2}\right)x+\alpha m_{l-1}\\
      &=\mu_1\gamma\beta^{-1}m_1-\mu_1\mu_2\beta^{-1}(q^2-1)m_0\end{align*}
      The others can also be checked to ensure $M_4(\underline{\mu})$ is indeed an $\textit{A}(\alpha,\beta,\gamma)$-module.
\begin{theo}\label{f3}
The module $M_4(\underline{\mu})$ is a simple $\textit{A}(\alpha,\beta,\gamma)$-module of dimension $l$.
\end{theo}
\begin{proof}
Note that the vector $m_k\in M_4(\underline{\mu})$ is an eigenvector of $zx$ associated with the eigenvalues $\displaystyle\frac{1-q^{-2k}}{1-q^{-2}}\beta$. With this fact, the proof is analogous to Theorem $\ref{f1}$.  
\end{proof}
\begin{theo}
    Let $\underline{\mu}=(\mu_1,\mu_2)$ and $\underline{\mu'}=(\mu'_1,\mu'_2)$ belong to ${\mathbb{K}}^2$. Then $M_4(\underline{\mu})\cong M_4(\underline{\mu}')$ if and only if $\mu_1={\mu'_1}$ and $\mu_2=\mu'_2$. 
\end{theo}
\subsection{Classification of simple $\textit{A}(\alpha,\beta,\gamma)$-modules}
Suppose $N$ is a simple $\textit{A}(\alpha,\beta,\gamma)$-module such that $z$ acts as nilpotent operator on $N$. By our assumption, the $\mathbb{K}$-space $\ker(z):=\{v\in N:vz=0\}$ is nonzero subspace of $N$. The central elements $x^l$ and $\Omega$ act as scalars on $N$, by Schur's lemma. So we can choose a nonzero vector $v \in \ker(z)$ such that 
\[vx^l=\zeta_1v,\ v\Omega=\zeta_2v\] for some $\zeta_1,\zeta_2 \in \mathbb{K}$. Then using the expression (\ref{omga}) of $\Omega$ we compute 
\[v\Omega=\zeta_2v \implies v\left(zd+\frac{\gamma x}{q^2-1}-\frac{\beta y}{q^2-1}\right)=\zeta_2v\implies \gamma (vx)-\beta (vy)=(q^2-1)\zeta_2v.\] 
We now claim that the vectors $vx^k,\ 0\leq k\leq l-1$ in $N$ are nonzero. If $\zeta_1\neq 0$, then we are done. If not, suppose $s$ be the smallest integer with $1 \leq s \leq l$ such that $vx^{s-1} \neq 0$ and $vx^s=0$. Then after simplifying the equality $(vx^s)z=0$, we have
\begin{align*}
      0=(vx^s)z=v\Big( q^{-2s}zx^s+\displaystyle\frac{1-q^{-2s}}{1-q^{-2}}\beta x^{s-1}\Big)=\displaystyle\frac{1-q^{-2s}}{1-q^{-2}}\beta vx^{s-1}.
    \end{align*}
    This implies $s$ is the smallest index with $1\leq s\leq l$ such that $q^{2s}=1$. Thus we have $s=l$ and the claim follows. 
Therefore the set $\{vx^k:0 \leq k \leq l-1\}$ consists of nonzero vectors of $N$. Let us choose
\[\mu_1:=\zeta_1, \ \mu_2:=\zeta_2\] so that 
$\underline{\mu}=(\mu_1,\mu_2) \in \mathbb{K}^2$. Now define a $\mathbb{K}$-linear map \[\Phi_4:M_4(\underline{\mu})\rightarrow N\] by specifying the image of the basis vectors of $M_4(\underline{\mu})$ only
\[\Phi_4(m_k):=vx^k\]One can easily verify that $\Phi_4$ is a non zero $\textit{A}(\alpha,\beta,\gamma)$-module homomorphism. In this verification the following computations will be very useful:
\[\left(vx^k\right)y=\begin{cases}
\gamma\beta^{-1}q^{2k}\left(vx^{k+1}\right)-\zeta_2\beta^{-1}q^{2k}(q^2-1)\left(vx^k\right)+\displaystyle\frac{1-q^{2k}}{1-q^2}\alpha\left(vx^{k-1}\right),&k \neq 0 \\
\gamma\beta^{-1}\left(vx\right)-\zeta_2\beta^{-1}(q^2-1)v,&k= 0.
\end{cases}\]
By Schur's lemma, $\Phi_4$ is an isomorphism because $M_4(\underline{\mu})$ and $N$ are both simple $\textit{A}(\alpha,\beta,\gamma)$-modules.
\par Finally the above discussions lead us to the final result of this section which provides an opportunity for the classification of simple $\textit{A}(\alpha,\beta,\gamma)$-modules in terms of scalar parameters.
\begin{theo}\label{main2}
Let $q^2$ be a primitive $l$-th root of unity. Then each simple $\textit{A}(\alpha,\beta,\gamma)$-module with nilpotent action of $z$ and $\beta \neq 0$ is isomorphic to the simple $\textit{A}(\alpha,\beta,\gamma)$-module $M_4(\underline{\mu})$ for some $\underline{\mu}:=(\mu_1,\mu_2)\in  \mathbb{K}^2$.
\end{theo}
 \section{Simple $\textit{A}(\alpha, \beta, \gamma)$-modules with nilpotent action of $z$ and $\beta=0,\ \gamma\neq 0$} 
 \subsection{Construction of simple $\textit{A}(\alpha,0,\gamma)$-modules} 
 In this subsection, we will construct simple $\textit{A}(\alpha,\beta,\gamma)$-modules with nilpotent action of $z$ and with the assumptions $\beta=0$ and $\gamma\neq 0$.
 \par\textbf{Simple modules of type $M_5(\underline{\mu})$:} For $\underline{\mu}:=(\mu_1,\mu_2) \in \mathbb{K}^2$, let us consider the $\mathbb{K}$-vector space $M_5(\underline{\mu})$ with basis $\{m_{k}:0 \leq k \leq l-1\}$. Define an $\textit{A}(\alpha, 0, \gamma)$-module structure on the $\mathbb{K}$-space $M_5(\underline{\mu})$ as follows: 
\begin{align*}
m_kx&=\begin{cases}
    \mu_2\gamma^{-1}q^{-2k}(q^2-1)m_k-q^{-2}\alpha\displaystyle\frac{1-q^{-2k}}{1-q^{-2}}m_{k-1},& k\neq 0\\
    \mu_2\gamma^{-1}(q^2-1)m_k,& k=0
\end{cases}\\
m_ky&=\begin{cases}m_{k+1},&k \neq l-1\\
    \mu_1m_0,&k= l-1.
\end{cases}\\
m_kz&=\begin{cases}
    \gamma\displaystyle\frac{1-q^{2k}}{1-q^{2}}m_{k-1},& k\neq 0\\
    0,& k=0.
\end{cases}
\end{align*}
\begin{theo}\label{f4}
The module $M_5(\underline{\mu})$ is a simple $\textit{A}(\alpha,0,\gamma)$-module of dimension $l$.
\end{theo}
\begin{proof}
Note that the vector $m_k\in M_5(\underline{\mu})$ is an eigenvector of $zy$ associated with the eigenvalues $\displaystyle\frac{1-q^{2k}}{1-q^{2}}\gamma$. With this fact, the proof is analogous to Theorem $\ref{f1}$.
\end{proof}
\begin{theo}
    Let $\underline{\mu}=(\mu_1,\mu_2)$ and $\underline{\mu'}=(\mu'_1,\mu'_2)$ belong to ${\mathbb{K}}^2$. Then $M_5(\underline{\mu})\cong M_5(\underline{\mu}')$ if and only if $\mu_1={\mu'_1}$ and $\mu_2=\mu'_2$. 
\end{theo}
\subsection{Classification of simple $\textit{A}(\alpha,0,\gamma)$-modules} Assume that $\gamma\neq 0$. Suppose $N$ is a simple $\textit{A}(\alpha,0,\gamma)$-module with nilpotent action of $z$. Note that the central elements $y^l$ and $\Omega$ act as scalars on $N$, by Schur's lemma. Therefore we can choose a nonzero vector $v \in \ker (z)$ such that 
\[vy^l=\zeta_1v,\ v\Omega=\zeta_2v\] for some $\zeta_1,\zeta_2 \in \mathbb{K}$. Then using the expression (\ref{omga}) of $\Omega$ along with the assumptions $\beta=0$ and $\gamma\neq 0$ we have
\[v\Omega=\zeta_2v \implies v\left(zd+\frac{\gamma x}{q^2-1}\right)=\zeta_2v\implies  vx=\zeta_2\gamma^{-1}(q^2-1)v.\] 
We now claim that the vectors $vy^k,\ 0\leq k\leq l-1$ in $N$ are nonzero. If $\zeta_1 \neq 0$ then we are done. If not suppose $s$ be the smallest integer with $1 \leq s \leq l$ such that $vy^{s-1} \neq 0$ and $vy^s=0$. Then after simplifying the equality $(vy^s)z=0$ we have
    \[0=(vy^s)z=v\Big( q^{2s}zy^s+\displaystyle\frac{1-q^{2s}}{1-q^{2}}\gamma y^{s-1}\Big)=\displaystyle\frac{1-q^{2s}}{1-q^{2}}\gamma vy^{s-1}.\]
    This implies $s$ is the smallest index with $1\leq s\leq l$ such that $q^{2s}=1$. Thus we have $s=l$ and the claim follows. Therefore the vectors $vy^k,\ 0 \leq k \leq l-1$ of $N$ are nonzero. Let us choose
\[\mu_1:=\zeta_1, \ \mu_2:=\zeta_2\] so that 
$\underline{\mu}=(\mu_1,\mu_2) \in \mathbb{K}^2$. Now define a $\mathbb{K}$-linear map \[\Phi_5:M_5(\underline{\mu})\rightarrow N\] by specifying the image of the basis vectors of $M_4(\underline{\mu})$ only
\[\Phi_5(m_k):=vy^k\]One can easily verify that $\Phi_5$ is a non zero $\textit{A}(\alpha,0,\gamma)$-module homomorphism. In this verification the following computations will be very useful:
\[\left(vy^k\right)x=\begin{cases}
    \zeta_2\gamma^{-1}q^{-2k}(q^2-1)\left(vy^k\right)-q^{-2}\alpha\displaystyle\frac{1-q^{-2k}}{1-q^{-2}}\left(vy^{k-1}\right),& k\neq 0\\
    \zeta_2\gamma^{-1}(q^2-1)v,& k=0
\end{cases}\]
By Schur's lemma, $\Phi_5$ is an isomorphism because $M_5(\underline{\mu})$ and $N$ are both simple $\textit{A}(\alpha,0,\gamma)$-modules.
\begin{theo}\label{main3}
Let $q^2$ be a primitive $l$-th root of unity. Then each simple $\textit{A}(\alpha,0,\gamma)$-module with nilpotent action of $z$ and $\gamma \neq 0$ is isomorphic to the simple $\textit{A}(\alpha,0,\gamma)$-module $M_5(\underline{\mu})$ for some $\underline{\mu}:=(\mu_1,\mu_2)\in \mathbb{K}^2$.
\end{theo}
\section{Simple $\textit{A}(\alpha, \beta, \gamma)$-modules with nilpotent action of $z$ and $\beta=\gamma=0$}
Firstly, the algebra $A(0,0,0)$ becomes a quantum affine space with three generators. The simple modules over such a quantum affine space have been classified in \cite{smsb}. So in the remainder of this section, we will assume that $\alpha\neq 0$.
 \subsection{Construction of simple $\textit{A}(\alpha,0,0)$-modules} Let us first start by constructing simple $\textit{A}(\alpha,0,0)$-modules with the assumption $\alpha \neq 0$.
 \par \textbf{Simple modules of type $M_6(\underline{\mu})$:} $\underline{\mu}:=(\mu_1,\mu_2)\in \mathbb{({K}^*)}^2$, let us consider the $\mathbb{K}$-vector space $M_6(\underline{\mu})$ with basis $\{m_{k}:0 \leq k \leq l-1\}$. Define an $\textit{A}(\alpha, 0,0)$-module structure on the $\mathbb{K}$-space $M_6(\underline{\mu})$ as follows: 
\begin{align*}
m_kx&=\begin{cases}
  \mu_1m_{k+1},&k\neq l-1\\
  \mu_1m_0,&k=l-1
\end{cases}\\
m_ky&=\begin{cases}
    \mu_1^{-1}\displaystyle\frac{q^{2k}\mu_2-\alpha}{q^2-1}m_{k-1},&k\neq 0\\
    \mu_1^{-1}\displaystyle\frac{\mu_2-\alpha}{q^2-1}m_{l-1},&k=0
\end{cases}\\
m_kz&=0.
\end{align*}
\begin{theo}\label{f5}
The module $M_6(\underline{\mu})$ is a simple $\textit{A}(\alpha,0,0)$-module of dimension $l$.
\end{theo}
\begin{proof}
Note that the vector $m_k\in M_6(\underline{\mu})$ is an eigenvector of $yx$ associated with the eigenvalues $\displaystyle\frac{q^{2k}\mu_2-\alpha}{q^2-1}$. With this fact, the proof is analogous to Theorem $\ref{f1}$.  
\end{proof}
\begin{theo}
    Let $\underline{\mu}=(\mu_1,\mu_2)$ and $\underline{\mu'}=(\mu'_1,\mu'_2)$ belong to ${(\mathbb{K}^*)}^2$. Then $M_6(\underline{\mu})\cong M_6(\underline{\mu}')$ if and only if $\mu_1^l={\mu'_1}^l$ and there exists some $i$ with $0 \leq i \leq l-1$ such that $\mu_2=q^{2i}\mu'_2$. 
\end{theo}
\par \textbf{Simple modules of type $M_7({\mu})$:} For $\mu \in \mathbb{K}$, let us consider the $\mathbb{K}$-vector space $M_7({\mu})$ with basis $\{m_{k}:0 \leq k \leq l-1\}$. Define an $\textit{A}(\alpha, 0,0)$-module structure on the $\mathbb{K}$-space $M_7({\mu})$ as follows: 
\begin{align*}
m_kx&=\begin{cases}
    -q^{-2}\alpha\displaystyle\frac{1-q^{-2k}}{1-q^{-2}}m_{k-1},& k\neq 0.\\
    0,& k=0.
\end{cases}\\
m_ky&=\begin{cases}
m_{k+1},& k\neq l-1\\
\mu m_0, & k=l-1
\end{cases}\\
m_kz&=0.
\end{align*}
\begin{theo}\label{f6}
The module $M_7({\mu})$ is a simple $\textit{A}(\alpha,0,0)$-module of dimension $l$.
\end{theo}
\begin{proof}
Note that the vector $m_k\in M_7({\mu})$ is an eigenvector of $yx$ associated with the eigenvalues $-q^{-2}\alpha\displaystyle\frac{1-q^{-2k}}{1-q^{-2}}$. Based on this fact, the proof is analogous to the proof of Theorem $\ref{f1}$.  
\end{proof}
\begin{theo}
    Let $\mu$ and ${\mu'}$ belong to ${\mathbb{K}}$. Then $M_7({\mu})\cong M_7({\mu}')$ if and only if $\mu={\mu'}$. 
\end{theo}
\begin{rema}
    The two types of simple $\textit{A}(\alpha,0,0)$-modules above are non-isomorphic. This is because the element $x$ acts on $M_6(\underline{\mu})$ as an invertible operator, whereas the element $x$ acts on $M_7({\mu})$ as a nilpotent operator.
\end{rema}
\subsection{Classification of simple $\textit{A}(\alpha,0,0)$-modules} Suppose $N$ is a simple $\textit{A}(\alpha,0,0)$-module with nilpotent action of $z$. Then the $\mathbb{K}$-space $\ker(z)=\{v\in N: vz=0\}$ is nonzero subspace of $N$. Now, based on the defining relations of $A(\alpha,0,0)$, we can verify that $\ker(z)$ is invariant under the action of $\textit{A}(\alpha,0,0)$. Since $N$ is a simple module, we conclude that $N=\ker(z)$, i.e., $Nz=0$.
\par Denote $g:=xy-yx$ in $\textit{A}(\alpha,0,0)$. Then the following commutation relations hold 
\begin{equation}\label{crel}
    gx=q^{-2}xg\ \ \text{and}\ \ gy=q^{-2}yg.
\end{equation}  Note that the elements $x^l,y^l$ and $g$ commute with each other in $\textit{A}(\alpha,0,0)$. 
So we can choose a common eigenvector $v \in \ker(z)=N$ of these commuting operators such that
\[vx^l=\zeta_1v,\ vg=\zeta_2v, vy^l=\zeta_3v\] for some $\zeta_1,\zeta_2,\zeta_3 \in \mathbb{K}$. 
\par First suppose that $\zeta_2=0$. Then the subspace $\ker(g)=\{v\in N: vg=0\}$ is nonzero and by (\ref{crel}), the $\ker(g)$ is invariant under the action of $\textit{A}(\alpha,0,0)$. Since $N$ is simple, we conclude that $N=\ker(z)=\ker(g)$. Thus in this case $N$ becomes a simple module over the factor algebra $\textit{A}(\alpha,0,0)/\langle z,g \rangle$ which is isomorphic to the commutative algebra $\mathbb{K}[x,y]/\langle(q^2-1)yx+\alpha \rangle$. Therefore such simple modules are one-dimensional and can be easily classified.
\par Now suppose that $\zeta_2 \neq 0$. Then we consider the following two cases, depending on the scalars $\zeta_1$ and $\zeta_3$.  
\par \textbf{Case I:} First assume that $\zeta_1\neq 0$. Then the set $\{vx^k:0 \leq k \leq l-1\}$ consists of nonzero vectors of $N$. Let us choose
\[\mu_1:=\zeta_1^{\frac{1}{l}},\ \mu_2:=\zeta_2\] so that $\underline{\mu}=(\mu_1,\mu_2)\in (\mathbb{K}^*)^2$. Now define a $\mathbb{K}$-linear map
\[\Phi_6: M_6(\underline{\mu})\rightarrow N\] by specifying the image of the basis vectors of $M_6(\underline{\mu})$ only
\[\Phi_6(m_k):=\mu_1^{-k}vx^k\]One can easily verify that $\Phi_6$ is a non zero $\textit{A}(\alpha,0,0)$-module homomorphism. In this verification the following computations will be very useful:
\[\left(vx^k\right)y=\begin{cases}
\displaystyle\frac{q^{2k}\zeta_2-\alpha}{q^2-1}\left(vx^{k-1}\right),&k \neq 0 \\
\zeta_1^{-1}\displaystyle\frac{\zeta_2-\alpha}{q^2-1}(vx^{l-1}),&k= 0.
\end{cases}\]
By Schur's lemma, $\Phi_6$ is an isomorphism because $M_6(\underline{\mu})$ and $N$ are both simple $\textit{A}(\alpha,0,0)$-modules.\\
\textbf{Case II:} Assume that $\zeta_1=0$. Since $\zeta_1 =0$ there exists $0\leq p\leq l-1$ such that $vx^{p}\neq 0$ and $vx^{p+1}=0$. Define $w:=vx^{p}\neq 0$. Note that $wy^l=\zeta_3w$.
\par We assert that the vectors $wy^k,\ 0\leq k\leq l-1$ in $N$ are non-zero. If $\zeta_3 \neq 0$, we are done. If not, let us assume that $s$ is the smallest integer with $1 \leq s \leq l$ such that $wy^{s-1} \neq 0$ and $wy^s=0$. Then after simplifying the equality $(wy^s)x=0$ we have
    \[0=(wy^s)x=w\Big( q^{-2s}xy^s-q^{-2}\displaystyle\frac{1-q^{-2s}}{1-q^{-2}}\alpha y^{s-1}\Big)=-q^{-2}\displaystyle\frac{1-q^{-2s}}{1-q^{-2}}\alpha wy^{s-1}.\]
    This implies $s$ is the smallest index with $1\leq s\leq l$ such that $q^{2s}=1$. Thus we have $s=l$ and the claim follows. 
\par It follows that the vectors $wy^k$, where $0 \leq k\leq l-1$ of $N$ are non-zero. Take ${\mu}=\zeta_3\in \mathbb{K}$. Now define a $\mathbb{K}$-linear map \[\Phi_7:M_7({\mu})\rightarrow N\] by specifying the image of the basis vectors of $M_7({\mu})$ only
\[\Phi_7(m_k):=wy^k\] One can easily verify that $\Phi_7$ is a non zero $\textit{A}(\alpha,0,0)$-module homomorphism. In this verification the following computations will be very useful:
\[\left(wy^k\right)x=\begin{cases}
-q^{-2}\displaystyle\frac{1-q^{-2k}}{1-q^{-2}}\alpha\left(wy^{k-1}\right),&k \neq 0 \\
0,&k= 0.
\end{cases}\]
By Schur's lemma, $\Phi_7$ is an isomorphism because $M_7({\mu})$ and $N$ are both simple $\textit{A}(\alpha,0,0)$-modules.
\par Finally the above discussions lead us to the final result of this section which provides an opportunity for the classification of simple $\textit{A}(\alpha,\beta,\gamma)$-modules in terms of scalar parameters.
\begin{theo}\label{main4}
Let $q^2$ be a primitive $l$-th root of unity. Then each simple $\textit{A}(\alpha,0,0)$-module with nilpotent action of $z$ and invertible action of $g$ is isomorphic to one of the simple $\textit{A}(\alpha,0,0)$-modules
\begin{itemize}
    \item $M_6(\underline{\mu})$ for some $\underline{\mu}:=(\mu_1,\mu_2)\in  {(\mathbb{K}^*)}^2$.
    \item $M_7({\mu})$ for some $\mu\in \mathbb{K}$.
\end{itemize}

\end{theo}

\end{document}